\documentclass[a4paper,reqno]{amsart}

\usepackage{amsmath}
\usepackage{mathptmx}
\usepackage{amssymb}
\usepackage{amsthm}
\usepackage{txfonts}
\usepackage{microtype}
\usepackage[usenames,dvipsnames]{xcolor}
\usepackage{enumitem}
	
\usepackage[colorlinks=true,linkcolor=Red,citecolor=Red,urlcolor=Red]{hyperref}
\usepackage[capitalise]{cleveref}
\usepackage{mathrsfs}
\usepackage{dsfont}
\usepackage[foot]{amsaddr}
\usepackage[utf8]{inputenc}
\usepackage{fullpage}

\renewcommand{\textit}[1]{\textsl{#1}}

\theoremstyle{plain}
\newtheorem{theorem}{Theorem}[section]
\newtheorem{lemma}[theorem]{Lemma}

\theoremstyle{definition}
\newtheorem{definition}[theorem]{Definition}
\newtheorem{remark}[theorem]{Remark}
\newtheorem{example}[theorem]{Example}

\setlist[itemize]{noitemsep, topsep=0em}

\newcommand{\N}{\mathbb{N}}
\newcommand{\Z}{\mathbb{Z}}
\newcommand{\R}{\mathbb{R}}
\newcommand{\C}{\mathbb{C}}
\newcommand{\TT}{\mathbb{T}}
\newcommand{\sx}{y}
\newcommand{\chf}[1]{{\chi}_{#1}}

\allowdisplaybreaks

\makeatletter
\let\save@mathaccent\mathaccent
\newcommand*\if@single[3]{%
  \setbox0\hbox{${\mathaccent"0362{#1}}^H$}%
  \setbox2\hbox{${\mathaccent"0362{\kern0pt#1}}^H$}%
  \ifdim\ht0=\ht2 #3\else #2\fi
  }
\newcommand*\rel@kern[1]{\kern#1\dimexpr\macc@kerna}
\newcommand*\widebar[1]{\@ifnextchar^{{\wide@bar{#1}{0}}}{\wide@bar{#1}{1}}}
\newcommand*\wide@bar[2]{\if@single{#1}{\wide@bar@{#1}{#2}{1}}{\wide@bar@{#1}{#2}{2}}}
\newcommand*\wide@bar@[3]{%
  \begingroup
  \def\mathaccent##1##2{%
    \let\mathaccent\save@mathaccent
    \if#32 \let\macc@nucleus\first@char \fi
    \setbox\z@\hbox{$\macc@style{\macc@nucleus}_{}$}%
    \setbox\tw@\hbox{$\macc@style{\macc@nucleus}{}_{}$}%
    \dimen@\wd\tw@
    \advance\dimen@-\wd\z@
    \divide\dimen@ 3
    \@tempdima\wd\tw@
    \advance\@tempdima-\scriptspace
    \divide\@tempdima 10
    \advance\dimen@-\@tempdima
    \ifdim\dimen@>\z@ \dimen@0pt\fi
    \rel@kern{0.6}\kern-\dimen@
    \if#31
      \overline{\rel@kern{-0.6}\kern\dimen@\macc@nucleus\rel@kern{0.4}\kern\dimen@}%
      \advance\dimen@0.4\dimexpr\macc@kerna
      \let\final@kern#2%
      \ifdim\dimen@<\z@ \let\final@kern1\fi
      \if\final@kern1 \kern-\dimen@\fi
    \else
      \overline{\rel@kern{-0.6}\kern\dimen@#1}%
    \fi
  }%
  \macc@depth\@ne
  \let\math@bgroup\@empty \let\math@egroup\macc@set@skewchar
  \mathsurround\z@ \frozen@everymath{\mathgroup\macc@group\relax}%
  \macc@set@skewchar\relax
  \let\mathaccentV\macc@nested@a
  \if#31
    \macc@nested@a\relax111{#1}%
  \else
    \def\gobble@till@marker##1\endmarker{}%
    \futurelet\first@char\gobble@till@marker#1\endmarker
    \ifcat\noexpand\first@char A\else
      \def\first@char{}%
    \fi
    \macc@nested@a\relax111{\first@char}%
  \fi
  \endgroup
}

\usepackage{scalerel,stackengine}
\stackMath
\newcommand\widecheck[1]{%
\savestack{\tmpbox}{\stretchto{%
  \scaleto{%
    \scalerel*[\widthof{\ensuremath{#1}}]{\kern-.6pt\bigwedge\kern-.6pt}%
    {\rule[-\textheight/2]{1ex}{\textheight}}
  }{\textheight}%
}{0.5ex}}%
\stackon[1pt]{#1}{\scalebox{-1}{\tmpbox}}%
}

\begin{document}

\title[Diffraction of return time measures]{Diffraction of return time measures}
\author[M.\ Kesseb\"ohmer]{M.\ Kesseb\"ohmer}
\address[M.\ Kesseb\"ohmer, A.\ Mosbach \& M.\ Steffens]{FB 3 -- Mathematik, Universit\"at Bremen, Bibliothekstr. 1, 28359 Bremen, Germany}
\author[A.\ Mosbach]{A.\ Mosbach}
\author[T.\,Samuel]{T.\,Samuel}
\address[T.\,Samuel]{\parbox[t]{38em}{Mathematics Department, California Polytechnic State University, San Luis Obispo, CA, USA and Institut Mittag-Leffler, Djursholm, Sweden}}
\author[M.\ Steffens]{M.\ Steffens}
\date{\today}


\maketitle

\begin{abstract}
Letting $T$ denote an ergodic transformation of the unit interval and letting $f\colon[0,1)\to \R$ denote an observable, we construct the $f$-weighted return time measure $\mu_\sx$ for a reference point $\sx\in[0,1)$ as the weighted Dirac comb  with support in $\Z$ and weights $f\circ T^z(\sx)$ at $z\in\Z$, and if $T$ is non-invertible, then we set the weights equal to zero for all $z < 0$. Given such a Dirac comb, we are interested in its diffraction spectrum which emerges from the Fourier transform of its autocorrelation and analyse it for the dependence on the underlying transformation. For certain rapidly mixing transformations and observables of bounded variation, we show that the diffraction of $\mu_{\sx}$ consists of a trivial atom and an absolutely continuous part, almost surely with respect to $\sx$.  This contrasts what occurs in the setting of regular model sets arising from cut and project schemes and deterministic incommensurate structures.  As a prominent example of non-mixing transformations, we consider the family of rigid rotations $T_{\alpha} \colon x \to x + \alpha \bmod{1}$ with rotation number $\alpha \in \R^+$. In contrast to when $T$ is mixing, we observe that the diffraction of $\mu_{\sx}$ is pure point, almost surely with respect to $\sx$.  Moreover, if $\alpha$ is irrational and the observable $f$ is Riemann integrable, then the diffraction of $\mu_{\sx}$ is independent of  $\sx$.  Finally, for a converging sequence $(\alpha_{i})_{i \in \N}$ of rotation numbers, we provide new results concerning the limiting behaviour of the associated diffractions. 
 
\end{abstract}

\section{Introduction and statement of main results}

Alloys with aperiodic long-range order were first discovered through diffraction experiments in the 1980's by Shechtman \textsl{el al.} \cite{PhysRevLett.53.1951} and subsequently by Ishimasa \textsl{et al.} \cite{PhysRevLett.55.511}.  Since then the theory of aperiodic order, also known as the mathematical theory of quasicrystals, has stimulated a tremendous amount of research, see for instance \cite{BaakeGrimm2013,MR3380566,MR1460016,MR1340198} and references therein. Indeed, the diffraction properties of quasicrystals, both physical and mathematical, are among their most striking features.

A mathematical idealisation of the set of atomic positions of a physical quasicrystal or incommensurate crystals is often given in terms of a measure $\mu$ supported on a locally compact Abelian group, for example lattices, for instance $\Z$.  Indeed, Hof \cite{MR1328260} established that the natural formulation of mathematical diffraction theory is via measures.  The diffraction of $\mu$ is given by the Fourier transform $\widehat{\gamma}$ of the autocorrelation $\gamma = \mu \circledast \widetilde{\mu}$, see \Cref{section:continuous_fourier_transform} for a precise definition.  Loosely speaking, the autocorrelation encodes the frequencies of distances between the atoms of $\mu$.  One of the most common ways to obtain $\widehat{\gamma}$ is via Bochner's Theorem, see for instance \cite{StrungaruRichard2017}.  An alternative approach, given in \cite{ArgabrightLamadrid1974,ArgabrightLamadrid1990}, is to compute the Fourier Bohr coefficients. 

The diffraction spectrum of Dirac combs supported on point sets such as regular model sets arising from cut and project schemes and deterministic incommensurate structures have been extensively studied and shown to be pure point, see for instance \cite{Baake-Grimm-2011,BaakeGrimm2013,MR2008926,StrungaruRichard2017} and references therein.  Diffraction of translation bounded weighted Dirac combs supported on locally compact Abelian group have also been studied in \cite{BaakeGrimm2013,MR2084582}.  Here we investigate spectral properties of a new class of weighted Dirac combs which we call weighted return time measures.  Letting $T$ denote a transformation of the unit interval, $\eta$ denote a $T$-invariant ergodic measure and $f \in L^{1}([0, 1], \eta)$ denote a non-negative observable, we define the $f$-weighted return time measure with respect to $T$ and with reference point $\sx \in [0, 1]$ by
\begin{align*}
\mu_{\sx} \coloneqq 
\begin{cases}
\displaystyle \ \sum_{n \in \N_0} f \circ T^{n} ( \sx )\ \delta_{n} & \text{if} \ T \ \text{is non-invertible,}\\[1.5em]
\displaystyle \ \sum_{z \in \Z} f \circ T^{z} ( \sx )\ \delta_{z} &  \text{if} \ T \ \text{is invertible.}
\end{cases}
\end{align*}
Here, for $z \in \mathbb{Z}$, we let $\delta_{z}$ denote the Dirac point mass at $z$.  In this note we answer the following questions.

First, in which way do mixing properties of $T$ have an impact upon spectral decomposition of the diffraction $\widehat{\gamma_{\mu_{\sx}}}$ of $\mu_{\sx}$. The conclusion we reach in \Cref{thm:difraction_mixing} is that in the case when $T$ is mixing and our observable $f$ is of bounded variation, the diffraction of $\mu_\sx$ consists of a trivial atom and an absolutely continuous part, for $\eta$-almost every $\sx$. This contrasts what occurs in the standard setting of regular model sets arising from cut and project schemes and deterministic incommensurate structures. To see how our setting fits into these latter settings see \Cref{Rem:RiemannCase}.

On the other hand, if $T_{\alpha} \colon x \mapsto x + \alpha \bmod 1$ is a rigid rotation with rotation number $\alpha \in \R^+$ -- the classical example of a non-mixing transformation of the unit interval -- the diffraction $\widehat{\gamma_{\mu_{\sx}}}$ of $ \mu_{\sx} =\mu_{\sx, \alpha}$ is pure point for $\eta$-almost every $\sx$, see \Cref{thm:diffraction_rotation}. Moreover, if $\alpha$ is irrational and the observable $f$ is Riemann integrable, then the diffraction of $\mu_{\sx}$ is independent of  $\sx$.  We remark that the Dirac combs we consider here, for rigid rotations, fit into the setting of \cite{BaakeGrimm2013,MR3723344}, and thus with some work, one may conclude this result from \cite{BaakeGrimm2013,MR3723344}. However, for completeness, we include a shorter proof tailored to our setting.

The second question stems from the work of \cite{MR2227822}. 
Let $\alpha\in\R^+$ and $(\alpha_{i})_{i \in \N}$ denote a positive convergent sequence of rotation numbers different from $\alpha$ with limiting value $\alpha$, let $(\sx_{i})_{i\in \N}$ denote a sequence of reference points, and let $(f_{i})_{i \in \N}$ denote a convergent sequence of Riemann integrable observables.  For $i \in \N$, let $\mu_{\alpha_i,\sx_i} = \mu_{\sx_{i}}$ denote the $f_{i}$-weighted return time measure with respect to $T_{\alpha_{i}}$ and with reference point $\sx_{i} \in [0, 1]$.  Here, we have that $\mu_{\alpha_{i}, \sx_{i}}$ converges in the vague topology, but does the sequence of diffractions $\widehat{\gamma_{\mu_{\sx_{i}}}}$ also converge in the vague topology? If so, what is the limiting measure?  The conclusion we reach in \Cref{thm:convergence} is the following.  The sequence $(\mu_{\alpha_i,\sx_i})_{i \in \N}$ of weighted Dirac combs induces a sequence of autocorrelations which has a vague limit $\gamma$.  However, we only have $\gamma = \gamma_{\mu_{\alpha,\sx}}$, for some $y \in [0, 1]$, if $\alpha$ is irrational. While this may seem counterintuitive at first sight, it reflects the fact that orbits of irrational rotations may be approximated by rational ones but not the other way around.  Note that the convergence of the associated diffractions follows from the continuity of the Fourier transform \cite{StrungaruRichard2017,Rudin1990_Fourier}.  See \Cref{fig2} for an illustration of the diffractions of two $f$-weighted return time measures associated to rigid rotations with irrational rotation numbers close together.

We remark that in \cite{BaakeLenz2004}, Baake and Lenz, using the work of Gouéré \cite{Gouere2002}, constructed a natural autocorrelation on the space of translation bounded measures under group actions. In \cite{LenzStrungaru2016} weakly almost periodic measures are considered, and shown to have a unique decomposition into a pure point diffractive part and continuous diffractive part. While the explicit decomposition of such measure is in general hard to obtain, by using Perron-Frobenius theory for mixing systems, we give such a decomposition in \eqref{eq:autocorralation_mixing_representartion}, also see \Cref{thm:difraction_mixing}. Also, by the results of \cite{LenzStrungaru2016}, we have that weighted return time measures for mixing transformations are, in general, not weakly almost periodic. Other works where the autocorrelation of group actions are investigated include \cite{Lenz2016,MuellerRichard2013,MR3723344}, and works discussing the effect of mixing conditions on the autocorrelation of tilings include \cite{BerendRadin1993,Mozes1989,Solomyak1996}.

\subsection*{Outline}

In \Cref{section:diffraction} we recall basic definitions and facts necessary to define the Fourier transform of a non-negative measure $\mu$ supported on a locally compact Abelian group $G$ and in the case when $G = \mathbb{Z}$ we define the autocorrelation and diffraction of $\mu$.  In \Cref{mixingSection} we provide an answer to our first question: in which way do mixing properties of $T$ have an impact upon the pure pointedness of the diffraction $\widehat{\gamma_{\mu_{\sx}}}$ of $\mu_{\sx}$.  Here our main results are \Cref{theorem:autocorrelation,thm:difraction_mixing}.  In \Cref{sec:ridig_rotations} we turn our attention to our second question which concerns the limiting behaviour of a sequence of diffractions.  Here, our main results are \Cref{thm:diffraction_rotation,thm:convergence}.  \Cref{mixingSection,sec:ridig_rotations} both conclude with a series of examples demonstrating the general theory developed in this note.

\section{General setup} \label{section:diffraction}

\subsection{Definitions and facts}

Although, in the sequel, we will predominately work in the case that $G = \Z$, below we state the necessary definitions and facts concerning the Fourier analysis of a measures supported on metrisable $\sigma$-compact locally compact Abelian groups $G$.  Given such a group $G$, we let $\omega_G$ denote the associated Haar measure. The space of complex-valued continuous function on $G$ is denoted by $\mathcal{C}(G)$ and is equipped with the topology of uniform convergence given by the supremum norm denoted by $\lVert \cdot \rVert_{\infty}$. For $\mathcal{D}(G)\subseteq\mathcal{C}(G)$ we denote by $\mathcal{D}_b(G)$ the subset of $\mathcal{D}(G)$ of bounded functions, by $\mathcal{D}_c(G)$ the subset of $\mathcal{D}(G)$ of compactly supported functions and by $\mathcal{D}^+(G)$ the subset of $\mathcal{D}(G)$ of real-valued non-negative functions.  For  $f\colon G\to\C$ and $x\in G$ set $\widetilde{f}(x)\coloneqq \widebar{f(-x)}$. We call $f$ \textsl{positive definite} if and only if, for all $N \in \N$, $x_{1}, \dots, x_{N} \in G$ and $c_{1}, \ldots, c_{N} \in \C$,
\begin{align*}
\sum_{1 \leq n, m \leq N} c_n \widebar{c_m}\ f(x_n-x_m) \geq 0.
\end{align*}
Examples of positive definite functions include characters and characteristic functions.

We denote the set of non-negative Radon measures with support contained in $G$ by $\mathscr{M}(G)$ and equip it with the vague topology. For $\mathscr{N}(G)\subseteq\mathscr{M}(G)$ denote by $\mathscr{N}_b(G)$ the subset of $\mathscr{N}(G)$ of bounded measures and by $\mathscr{N}_c(G)$ the subset of $\mathscr{N}(G)$ of measures with compact support. For $\mu \in \mathscr{M}(G)$, let $L^1(G,\mu)$ represent the Banach $*$-algebra of $\mu$-integrable functions with bounded $\mu$-$L^1$-norm denoted by $\lVert \cdot \rVert_{1}$ and, for $f,g\in L^1(G, \mu)$, define the \textsl{convolution} of $f$ with $g$ by
\begin{align*}
f*g(y) &\coloneqq \int f(x)g(y-x)\ \mathrm{d}\mu(x).
\intertext{Note that $L^1(G, \mu)$. When $\mu = \omega_G$ we write $L^1(G)$ for $L^1(G,\omega_G)$.  For two measures $\mu, \nu\in\mathscr{M}(G)$, the \textsl{convolution} of $\mu$ with $\nu$ is defined by}
\mu*\nu(A) &\coloneqq \int \chf{A}(x+y)\ \mathrm{d}\mu(x)\ \mathrm{d}\nu(y)
\end{align*}
for all Borel sets $A\subseteq G$.  Here, for a subset $A$ of $G$, we denote by $\chf{A}$ the characteristic function on $A$. Namely, $\chf{A}(x) = 0$ if $x \not\in A$ and $\chf{A}(x) = 1$ if $x \in A$.  The convolution of two measures is again a measure, but may not necessarily be Radon. However, if $\mu,\nu \in \mathscr{M}_b(G)$, then $\mu*\nu \in \mathscr{M}_b(G)$.  

In order to define the autocorrelation, and hence the diffraction, of a measure, we make use of the following.  For $\mu \in \mathscr{M}(G)$ and $g\in L^1(G,\mu)$, set
\begin{align*}
\langle \mu, g \rangle \coloneqq \int g \ \mathrm{d} \mu
\end{align*}
and let $\widetilde{\mu}$ denote the unique measure satisfying $\langle \widetilde{\mu}, \widebar{g}\rangle = \langle \mu, \widetilde{g}\rangle$. A measure $\mu$ is called \textsl{positive definite} if and only if $\langle \mu, f*\widetilde{f} \rangle \geq 0$ for all $f\in\mathcal{C}_c(G)$. The set of positive definite measures form a closed convex cone, see \cite{BergForst1975}, and is denoted by $\mathscr{M}_p(G)$.  Notice, if $\mu\in\mathscr{M}(G)$ is such that $\mu*\widetilde{\mu}\in\mathscr{M}(G)$, then $\mu*\widetilde{\mu}$ is positive definite, see for instance \cite{BergForst1975,Rudin1990_Fourier}.

The dual group of $G$ is denoted by $\Gamma$ and is a locally compact Abelian group. We call the elements of $\Gamma$ \textsl{characters} and set $(x,\gamma) \coloneqq \gamma(x)$, for $x \in G$ and $\gamma \in \Gamma$. Note, by Pontryagin's duality theorem, the dual group of $\Gamma$ is isomorphic to $G$. The Fourier transform is a linear norm-decreasing mapping from $(L^1(G), \lVert \cdot \rVert_{1})$ to $(\mathcal{C}_0(\Gamma), \lVert \cdot \rVert_{\infty})$ given by
\begin{align*}
\widehat{f}(\gamma) \coloneqq \int f(x) \widebar{(x,\gamma)}\ \mathrm{d}\omega_G(x),
\end{align*}
for $f \in L^1(G)$.  Here $\mathcal{C}_0(\Gamma)$ denotes the closure of $\mathcal{C}_c(\Gamma)$ in $\mathcal{C}(\Gamma)$.  The Fourier transform $\widehat{\mu}$ of a measure $\mu \in \mathscr{M}_p(G)$ is the unique measure satisfying $\langle \mu, \widehat{f} \rangle = \langle \widehat{\mu}, f \rangle$ for all $f\in \mathcal{P}(\Gamma)\coloneqq \{g*\widetilde{g} \colon g\in\mathcal{C}_c(\Gamma)\}$.  We refer the reader to \cite{BergForst1975} for a proof that $\widehat{\mu}$ is well-defined.  The Fourier transformation of measure is a continuous operation, see for instance  \cite{StrungaruRichard2017,Rudin1990_Fourier}.  In \cite{BergForst1975}, it is also shown that if $\mu\in\mathscr{M}_p(G)$, then $\mu$ is translation bounded; that is, for $K\subseteq G$ a compact set $\sup\{\mu(g+K) : g\in G\}$ is finite. Additionally, if $\mu$ is positive definite and bounded, then there exists a density $h\in\mathcal{C}(\Gamma)$ such that $\widehat{\mu} = h\omega_\Gamma$.

\subsection{Autocorrelation and diffraction} \label{section:continuous_fourier_transform}

Unless otherwise stated, from here on, we assume that $G = \Z$ , in this case its dual group $\Gamma$ is isomorphic to the unit circle in $\C$, which is denoted by $\TT$, and its Haar measure $\omega_{\Z}$ is the counting measure.  Namely, $\omega_{\Z} \coloneqq \sum_{z \in \Z} \delta_{z}$, where $\delta_{z}$ denotes the Dirac point mass at $z$; that is $\delta_{z}(A) = 0$ if $z \not\in A$ and $\delta_{z}(A) = 1$ if $z \in A$, for $A$ a Borel set.  For $n \in \N$, we set $B_n \coloneqq \{x\in \Z \colon \lvert x \rvert \leq n\}$ and for a Borel set $K\subseteq \Z$ we set $\mu\vert_{n}(K)\coloneqq \mu(K\cap B_n)$. Let $\mu\in\mathscr{M}(\Z)$, if the sequence
\begin{align*}
(\gamma_{\mu,n})_{n\in\N} \coloneqq \left(\frac{\mu\vert_{n}*\widetilde{\mu\vert_{n}}}{\omega_{\Z}(B_n)} \right)_{n\in\N}
\end{align*}
attains a unique vague limit $\mu\circledast\widetilde{\mu}$ in $\mathscr{M}(\Z)$, then this limit is called the \textsl{autocorrelation} of $\mu$ and is denoted by $\gamma_{\mu}$.  Here the symbol $\circledast$ is referred to as the \textsl{Eberlein convolution}, see \cite{BaakeGrimm2013}. By a result of \cite{Schlottmann1999}, one may show, in the case when $\mu$ is translation bounded, that
\begin{align}\label{eq:autocorrolation_Schlottmann1999}
\mu\circledast\widetilde{\mu} = \lim_{n\to\infty} \frac{\mu*\widetilde{\mu\vert_{n}}}{\omega_{\Z}(B_n)}.
\end{align}
As an aside, let us note that the above construction can be performed for general local compact $\sigma$-compact Abelian groups.  In this more general setting, the sequence $(B_n)_{n\in\N}$ above is replaced by an arbitrary Van-Hove sequence, see \cite{MuellerRichard2013,Schlottmann1999}.

Since $\mathscr{M}_p(G)$ is closed, if it exists, the autocorrelation $\gamma_\mu$ is a positive definite measure, namely $\gamma_\mu \in \mathscr{M}_p(G)$. With this at hand, we may define the \textsl{diffraction} of a measure $\mu$ to be the Fourier transform $\widehat{\gamma_{\mu}}$ of its autocorrelation $\gamma_\mu$ which is also the vague limit of the sequence $(\widehat{\gamma_{\mu,n}})_{n\in\N}$.

\begin{remark}
Let $\mu$ be a measure on $\R$ with support contained in $\Z$. Then the restriction $\mu|_{\Z}$  defines a measure on $\Z$, but can be lifted back to $\R \cong \R / \Z \ltimes \Z$ via 
\begin{align*}
\mu(f)= (\mu|_\Z*\delta_0)(f)\coloneqq \int_{\R / \Z} \int_{\Z} f(x+y) \ \mathrm{d}\mu|_{\Z}(x)\ \mathrm{d}\delta_{0}(y).
\end{align*}
An elementary calculation shows $\gamma_{\mu} = \mu \circledast \widetilde{\mu} = (\mu|_\Z*\delta_0)\circledast\widetilde{(\mu|_\Z*\delta_0)} = (\mu|_\Z\circledast\widetilde{\mu|_\Z})*\delta_0 = \gamma_{\mu|_\Z}*\delta_0$. Since $\delta_0$ is a measure on $\R/\Z$, an application of the Poisson-summation formula yields $\widehat{\gamma_{\mu}} =\widehat{\gamma_{\mu|_\Z}*\delta_0} = \widehat{\gamma_{\mu|_\Z}}*\widehat{\delta_0} = \widehat{\gamma_{\mu|_\Z}}*\omega_\Z$. This shows that no extra information is gained by considering the Dirac comb on $\R$ instead of $\Z$. 
\end{remark}

\section{Autocorrelation of mixing transformations of the unit interval} \label{mixingSection}

Here we investigate the autocorrelation and diffraction of Dirac combs emerging from ergodic transformations of the unit interval. To this end, let $T\colon[0,1]\to[0,1]$ and let $\eta$ denote a $T$-invariant ergodic Borel measure. Recall that $\eta$ is $T$-invariant if for all Borel sets $A\subseteq[0,1]$ we have $\eta(T^{-1}(A)) = \eta(A)$, and that $\eta$ is ergodic if $T^{-1}(A) = A$, then $\eta(A) = 0$ or $\eta([0,1] \setminus A) = 0$.

\begin{definition}\label{Def:mu}
Assume the above setting and let $f\in L^1([0,1],\eta)$ be a real-valued non-negative bounded function. We define the \textsl{$f$-weighted return time measure} $\mu_{\sx}$ with respect to $T$ and with \textsl{reference point} $\sx \in [0,1]$ by
\begin{align*}
\mu_{\sx} \coloneqq 
\begin{cases}
\displaystyle \ \sum_{n \in \N_0} f \circ T^{n} ( \sx )\ \delta_{n} & \text{if} \ T \ \text{is non-invertible,}\\[1.5em]
\displaystyle \ \sum_{z \in \Z} f \circ T^{z} ( \sx )\ \delta_{z} &  \text{if} \ T \ \text{is invertible.}
\end{cases}
\end{align*}
\end{definition}

The following theorem ensures, for almost every $\sx \in [0,1]$, that the \textsl{autocorrelation} $\gamma_{\mu_\sx} = {\mu}_{\sx}\circledast\widetilde{\mu}_{\sx}$ of an $f$-weighted return time measure $\mu_\sx$ exists and give its almost sure value.  Before stating the result we introduce the following notation.  For $z \in \Z$, we set
\begin{align*}
\Xi(T,\eta)(z) \coloneqq \begin{cases}
\displaystyle\frac{1}{2} \int f \circ T^{\lvert z \rvert}\cdot f \ \mathrm{d}\eta & \text{if} \ T \ \text{is non-invertible,}\\[1.5em]
\displaystyle\int f \circ T^{-z}\cdot f \ \mathrm{d}\eta & \text{if} \ T \ \text{is invertible.}
\end{cases}
\end{align*}

\begin{theorem}\label{theorem:autocorrelation}
Assuming the setting of \Cref{Def:mu} the autocorrelation $\gamma_{\mu_\sx}$ exists for $\eta$-almost every $\sx$ 
and equals
 \begin{align*}
\gamma_{\mu_{\sx}} &= \sum_{z \in \Z}\Xi(T,\eta)(z)\ \delta_z.
\end{align*}
\end{theorem}

\begin{proof}
We will prove the statement in the case that $T$ is non-invertible as the case when $T$ is invertible follows analogously.  For every $\varphi\in\mathcal{C}_c(\Z)$, by \eqref{eq:autocorrolation_Schlottmann1999},
\begin{align*}
\begin{aligned}
\langle \gamma_{\mu_\sx}, \varphi\rangle 
&=  \lim_{N \to \infty} \ \frac{1}{2N} \left\langle {\mu_\sx}\ast \widetilde{{\mu_\sx}}\vert_N, \varphi \right\rangle\\
&= \lim_{N\to\infty} \ \frac{1}{2N} \int \int \chf{[-N,N]}(n)\ \varphi(m+n)\ \mathrm{d}{\mu_\sx}(m)\ d\widetilde{{\mu_\sx}}(n) \\
&= \lim_{N\to\infty} \ \frac{1}{2N} \int \overline{\int \chf{[-N,N]}(-n)\  \widebar{\varphi(m-n)}\ \mathrm{d}{\mu_\sx}(m)\ \mathrm{d}{\mu_\sx}(n)} \\
&= \lim_{N\to\infty} \ \frac{1}{2N} \sum_{m\in\N_0} \sum_{0 \leq n \leq N} \varphi(m-n)\ f \circ T^{n}(\sx) \cdot f \circ T^{m}(\sx) \\
&= \lim_{N \to \infty} \ \frac{1}{2N} \ \sum_{z \in \Z} \varphi(z) \sum_{\max\{0,-z\} \leq n \leq N} f \circ T^{n}(\sx) \cdot f \circ T^{z+n}(\sx) \\
&= \lim_{N \to \infty} \frac{1}{2N} \left( \ \sum_{z<0} \ \varphi(z) \sum_{-z\leq n\leq N} (f \circ T^{-z} \cdot f )( T^{z+n}(\sx)) + \sum_{z\geq0} \ \varphi(z) \sum_{0\leq n \leq N} ( f \cdot f \circ T^{z} )( T^{n}(\sx)) \right)\\
&= \sum_{z<0} \ \varphi(z) \lim_{N \to \infty}  \frac{1}{2(N+z)}  \sum_{-z\leq n\leq N} ( f \circ T^{-z} \cdot f )( T^{z+n}(\sx)) + \sum_{z\geq0} \ \varphi(z) \lim_{N \to \infty} \frac{1}{2N}  \sum_{0\leq n \leq N} ( f \cdot f \circ T^{z} )( T^{n}(\sx)).
\end{aligned}
\end{align*}
For  $\eta$-almost every $\sx$, we may apply Birkhoff's ergodic theorem to obtain
\begin{equation*}
\langle\gamma_{\mu_\sx},~\varphi\rangle
= \frac{1}{2}\left( \ \sum_{z<0} \ \varphi(z) \ \Xi(T, \eta)(z) + \sum_{z\geq0} \ \varphi(z) \ \Xi(T, \eta)(z) \right). \qedhere
\end{equation*}
\end{proof}

\begin{remark}\label{uniquelyErgodic}
If $T$ is uniquely ergodic and the observable $f$ is continuous, then the autocorrelation exists for every $\sx\in [0,1]$ and the limit exists uniformly in $\sx$.  Such transformations include rigid rotations with irrational rotation number \cite{Walters1982} and minimal interval exchange maps satisfying Boshernitzan's property P condition \cite{MR808101}.
\end{remark}

\begin{remark}
Using the ergodic theorem of Lindenstrau{\ss}, one can show the existence of the autocorrelation for invariant measures of certain group actions \cite{MuellerRichard2013}.
\end{remark}

In addition to the assumptions of \Cref{Def:mu}, in the remainder of this section, we assume that $T$ is a piecewise monotonic transformation of the unit interval which is mixing with respect to $\eta$. By piecewise monotonic, we mean there exists a finite partition $\mathscr{I}$ of $[0,1]$ such that for all $I\in\mathscr{I}$ the restriction of $T$ on $I$, written as $T\vert_{I}$, is continuous, strictly monotone and differentiable. Recall that a transformation $T$ is mixing with respect to $\eta$, if for all Borel sets $A, B \subseteq [0, 1]$, we have 
\begin{align*}
\lim_{n \to \infty} \eta(T^{-n}(A) \cap B) = \eta(A)\eta(B).
\end{align*}
Note that the property of mixing implies ergodic, see \cite{Walters1982}.  The \textsl{variation} $\operatorname{var}(f)$ of an observable $f \colon [0, 1] \to \R$ is defined by
\begin{align*}
\operatorname{var}(f) \coloneqq \inf_{g = f \; \text{a.s.}} \sup \left\{ \sum_{1 \leq i \leq N} \lvert g(x_i)-g(x_{i-1}) \rvert \colon N \geq 1\; \text{and} \; 0 \leq x_0 < x_1 < \dots < x_N \leq 1 \right\}.
\end{align*}
We say $f \colon [0,1] \to \R$ is of bounded variation if $\operatorname{var}(f) < \infty$.  The space of real-valued integrable functions of bounded variation will be denoted by $\operatorname{BV} \coloneqq \{ f \in L^{1}([0,1],\eta) \colon \operatorname{var}(f) < \infty \}$ and we equip this space with the norm $\lVert f \rVert_{\textup{var}} \coloneqq \max\{ \lVert f \rVert_{1}, \operatorname{var}(f) \}$.  Let $P \colon\!\!\operatorname{BV} \to \operatorname{BV}$ denote the Perron-Frobenius-operator given by
\begin{align*}
P(s)(x) \coloneqq \sum_{y\in T^{-1}(x)} \phi(y)\cdot s(y) = \sum_{I\in\mathscr{I}} \phi\circ T\vert_{I}^{-1}(x)\cdot s\circ T\vert_{I}^{-1}(x)\cdot\chf{_{T(I)}}(x),
\end{align*}
for $s \in \operatorname{BV}$  and where $\phi\colon[0,1]\to(0,\infty)$ denotes the geometric potential function; namely, for all $x\in \Omega_0\coloneqq \bigcup_{I\in\mathscr{I}} \operatorname{I}^\circ$, we set $\phi(x)\coloneqq 1/|T'(x)|$ and, for all $x\in[0,1]\backslash\Omega_0$, we set $\phi(x) \coloneqq \lim_{y\to x} \, \inf \{ \phi(y) \colon y\in\Omega_0 \}$.  Note, the dual of the Perron-Frobenious-operator preserves the Lebesgue measure, denoted by $\Lambda$.
 It is known that the operator $P$ has a simple maximal eigenvalue $\lambda = 1$, see \cite{Keller1984,HofbauerKeller1982}. We denote the associated eigenfunction by $h$ and recall that it is positive. Moreover, $P = \Phi + \Psi$, where $\Phi(f) = \langle \Lambda, f \rangle \, h$ and where $\Psi$ is such that there exists an $M>0$ and $q\in(0,1)$ with $\|\Psi^n\|_{\operatorname{op}}\leq Mq^n$. Additionally, $\Phi\circ\Psi = \Psi\circ\Phi = 0$ and so $P^n = \Phi + \Psi^n$, for all $n\in\N$. 

With the above at hand we can explicitly compute the diffraction of $\gamma_{\mu_\sx}$.  

\begin{theorem}\label{thm:difraction_mixing}
Let $T$ be a piecewise monotonic transformation of the unit interval which is mixing with respect to $\eta$, $f \in \operatorname{BV}$ be non-negative, $\sx \in [0, 1]$ and $\mu_\sx$ denote the $f$-weighted return time measure with respect to $T$ and with reference point $\sx$.  The diffraction of $\mu_\sx$ is given, for $\eta$-almost every $\sx$, by
\begin{align*}
\widehat{\gamma_{\mu_\sx}} = \frac{1}{2} \left( \int f\ \mathrm{d}\eta \right)^2 \delta_1 + g\,\omega_\TT.
\end{align*}
Here $g(x) \coloneqq \sum_{z\in\Z} (c_z/2 )\ (x,z)$, where $c_z \coloneqq \int \Psi^{\lvert z \rvert}(f\cdot h) f\ \mathrm{d}\Lambda $, for $z\neq 0$, and $c_0\coloneqq \int \lvert f \rvert^2 \mathrm{d}\eta - \left( \int f\ \mathrm{d}\eta \right)^2$.
 \end{theorem}

\begin{proof}
In \cite{MR594335} it is shown that if $T \colon [0, 1] \to [0, 1]$ is a piecewise continuous injective map, then $T$ is not mixing with respect to any Borel measure.  Therefore, by our hypothesis, we may assume that $T$ is non-invertible.

For $n\in\N$, recalling that the Haar measure $\omega_\TT$ is the normalised Lebesgue measure, we observe the following chain of equalities.
\begin{align*}
\Xi(T,\eta)(n)
= \int P^n(h \cdot f )\ f \ \mathrm{d}\Lambda
= \int \left(h \int f \cdot h\ \mathrm{d}\Lambda + \Psi^n(f \cdot h)\right)\ f\ \mathrm{d}\Lambda
= \left( \int f \cdot h\ \mathrm{d}\Lambda \right)^{2} + \int \Psi^n(f \cdot h)\ f\ \mathrm{d}\Lambda 
\end{align*}
This in tandem with \Cref{theorem:autocorrelation} allows us to write the autocorrelation $\gamma_{\mu_\sx}$ as
\begin{align}\label{eq:autocorralation_mixing_representartion}
2 \gamma_{\mu_\sx} 
= \left( \int f\ \mathrm{d}\eta \right)^2 \sum_{z\in\Z} \delta_z + \sum_{z\in\Z} c_z \delta_z.
\end{align}
For $\varphi = l * \widetilde{l} \in \mathcal{P}(\TT)$, we have that
\begin{align*}
\langle 2\gamma_{\mu_\sx}, \widehat{\varphi} \rangle 
=& \left( \int f\ \mathrm{d}\eta \right)^2 \sum_{z\in\Z}\ \widehat{\varphi}(z) + \sum_{z\in\Z}\ c_z\ \widehat{\varphi}(z) \\
=& \left( \int f\ \mathrm{d}\eta \right)^2 \sum_{z\in\Z}\ \widehat{\varphi}(z)\ (1,z) + \sum_{z\in\Z}\ c_z \int \varphi(x)\ (x,z)\ \mathrm{d}\omega_\TT(x) \\
=& \left( \int f\ \mathrm{d}\eta \right)^2 \varphi(1) + \int \varphi(x) \sum_{z\in\Z}\ c_z\ (x,z)\ \mathrm{d}\omega_\TT(x) 
= \left\langle \left( \int f\ \mathrm{d}\eta \right)^2 \delta_{1} + g\ \omega_\TT, \varphi \right\rangle.
\end{align*}
To split the series in the first equality we require that both series on the right-hand-side are absolutely convergent. This is true for the first series, since $\widehat{\varphi}(z) = \lvert \, \widehat{l}(z) \, \rvert^{2}$ and thus $\|\widehat{\varphi}\|_1 =  \sum_{z\in\Z}\ \widehat{\varphi}(z) = \sum_{z\in\Z}\ \widehat{\varphi}(z) (\mathds{1},z) =  \varphi(\mathds{1}) < \infty$.  To see that the second series is absolutely convergent, notice $(c_z)_{z\in\N}$ and $(c_{-z})_{z\in\N}$ are sequences of exponential decay and $\varphi$ is a continuous function on a compact space. With this at hand, we note that the integral and the sum of the second component in the second equality can be interchanged by Lebesgue's dominated convergence theorem.
\end{proof}

\begin{example}\label{Exmp:Mixing}
Let for $k\in\N_{\geq2}$ $T_k \colon [0, 1)\to [0, 1)$ be given by $T_k(x) = k x \bmod{1}$ and set $f(x) = x$. An elementary calculation shows that $T_k$ is a piecewise monotonic transformation and mixing with respect to $\Lambda$. Hence for $n\in\N$,
\begin{align*}
2\ \Xi(T,\eta)(n) = 
\int f \cdot f\circ T_k^n\mathrm{d}\Lambda 
= \sum_{m=0}^{k^n-1} \int_{m k^{-n}}^{(m+1)k^{-n}} \hspace{-0.7em} x (k^nx-m)\ \mathrm{d}x
= \frac{k^{-2n}}{6} \sum_{m=0}^{k^n-1} (3m+2)
= \frac{1}{4} \frac{k^n+ 1/3}{k^n}.
\end{align*}
Using \Cref{theorem:autocorrelation} and \eqref{eq:autocorralation_mixing_representartion}, this gives, for $z\neq 0$,
\begin{align*}
2\ \Xi(T_k,\eta)(z) = \frac{1}{4} \frac{k^{\lvert z \rvert}+1/3}{k^{\lvert z \rvert}} 
 = \frac{1}{4} + \int \Psi^{\lvert z \rvert}(f) \cdot f\ \mathrm{d}\Lambda.
\end{align*}
Hence,
\begin{align*}
2\gamma_{\mu_\sx} = \frac{1}{4} \sum_{z\in\Z\setminus \{0\}} \delta_z + \sum_{z\in\Z\setminus \{0\}} \frac{1}{4} \left( \frac{k^{\lvert z \rvert}+1/3 }{k^{\lvert z \rvert}}-1 \right) \delta_z + \frac{1}{3} \delta_0 = \frac{1}{4} \sum_{z\in\Z} \delta_z + \frac{1}{12} \sum_{z\in\Z} k^{-\lvert z \rvert} \delta_z.
\end{align*}
Combining this with \Cref{thm:difraction_mixing}, yields $\widehat{\gamma_{\mu_\sx}} = \delta_{1}/8 + g_k\ \omega_{\TT}$,  where 
\begin{align*}
g_k(x) = \sum_{z\in\Z} k^{-|z|} (x,z)/24 = \frac{k-k^{-1}}{k+k^{-1}-2\cos(2\pi x)}.
\end{align*}
See \Cref{fig1} for the graph of $g_{k}$ for different values of $k$.
\begin{figure}
  \centering
    \includegraphics[width=0.5\textwidth]{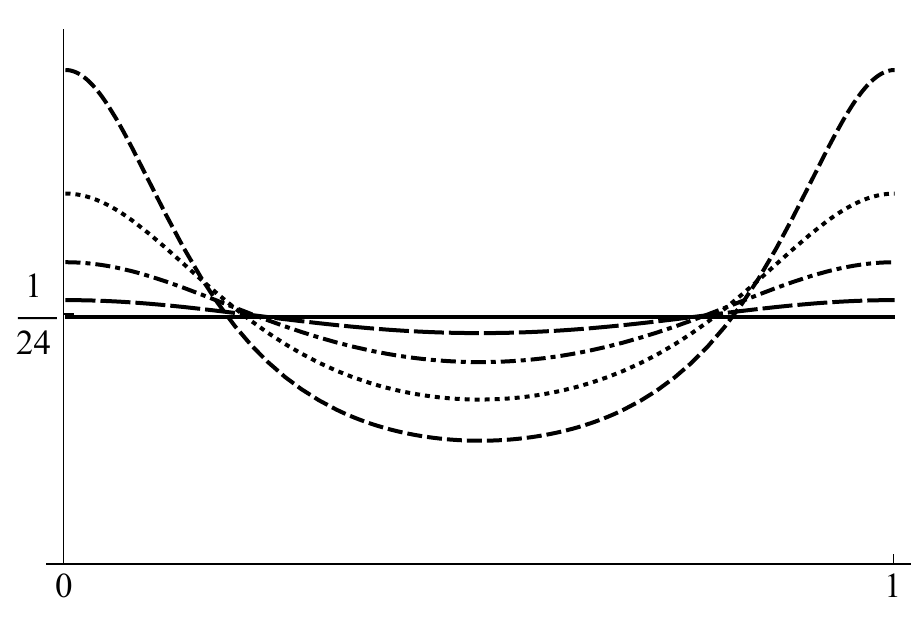}
      \caption{The density $g_{k}$ for $k=3,5,10,30$ of the diffraction measure $\gamma_{\mu_\sx}$, as determined in \Cref{Exmp:Mixing}, are approximating the constant density of hight $1/24$. This indicates that the decay of correlation for the observable $f$ decays faster for larger values of  $k\in \N_{\geq 2}$.}\label{fig1}
\end{figure}
\end{example}

\section{Rigid rotations}\label{sec:ridig_rotations}

We continue our exploration by considering dynamics given by a rotation of the unit interval.  Namely, for $\alpha \in \R^{+}$, we define $T_{\alpha} \colon [0, 1) \to [0, 1)$ by $T_{\alpha}(x) \coloneqq \{ x + \alpha \}$ and consider the dynamical system $([0, 1), T_{\alpha})$.  Here, $\{ t \}$ denotes the fractional part of $t \in \R$.  Note, the transformation $T_{\alpha}$ is topologically conjugate to the rotation map $z \mapsto \operatorname{e}^{2\pi i  \alpha} z$ on the unit circle $\TT$ in $\C$, where the conjugating map $\iota \colon [0, 1) \to \TT$ is given by $\iota(x) = \operatorname{e}^{2 \pi i x}$, for $x \in [0, 1)$.  In the case that $\alpha$ is irrational, the transformation $T_{\alpha}$ is uniquely ergodic, where the unique ergodic measure $\eta_{\alpha}$ is the Lebesgue measure $\Lambda$. On the other hand, if $\alpha = p/q$ with $p, q \in \N$ and $\operatorname{gcd}(p, q) = 1$, for each $w \in [0, 1)$, the measure $\eta_{\alpha} = \eta_{q, w} \coloneqq q^{-1} \sum_{k\in\Z_q} \delta_{\{ w + k/q \}}$ is an ergodic measure for $T_{\alpha}$.  In both cases $T_{\alpha}$ is not mixing with respect to $\eta_{\alpha}$ and so does not belong to the setting of \Cref{thm:difraction_mixing}.

For $q \in \N$, let $\Z_{q}$ denote the set $\{ 0, 1, \dots, q-1 \}$ equipped with the binary operation of addition modulo $q$ and, for $m \in \N$, let $[m]_{q}$ denote the unique element belonging to the intersection $\Z_{q} \cap \{ m + n q \colon n \in \N \}$.  Observe that $\Z_{q}$ is a locally compact Abelian group.

Let $\alpha=p/q$ with $p, q \in \N$ and $\gcd(p,q) = 1$. For $w \in [0, 1)$ and $f \in L^{1}([0, 1), \eta_{q,w})$ we have $\operatorname{supp}(\eta_{q,w})\cong \Z_q$ and $\Xi(T,\eta_{q,w})(z) = \Xi(T_\alpha,\eta_{q,w})([z]_{q})$ for $z\in\Z$. Further, for every $\sx\in \operatorname{supp}(\eta_{q,w})$, \Cref{theorem:autocorrelation} yields
\begin{align}
\gamma_{\mu_\sx} &= \sum_{z\in\Z}\Xi(T_\alpha,\eta_{q,w})(z)\ \delta_z,
\label{eq:theorem:autocorrelation_1}
\intertext{where $\mu_\sx$ denotes the $f$-weighted return time measure with respect to  $T_{\alpha}$ and with reference point $\sx$. 
If $\alpha\in\R^+\setminus\mathbb{Q}$, then $\operatorname{supp}(\eta_\alpha) = \operatorname{supp}(\Lambda) = [0,1)$, combining this with  and \Cref{theorem:autocorrelation} yields, for $f \in L^{1}([0, 1), \Lambda)$ and $\Lambda$-almost every $\sx$,}
\gamma_{\mu_\sx} &= \sum_{z\in\Z}\Xi(T_\alpha,\Lambda)(z)\ \delta_z,
\label{eq:theorem:autocorrelation_2}
\end{align}
where again $\mu_\sx$ denotes the $f$-weighted return time measure with respect to $T_{\alpha}$ and with reference point $\sx$. 

\begin{remark} \label{Rem:RiemannCase}
If $\alpha$ is irrational and $f$ is Riemann integrable, then the autocorrelation exists for every reference point $\sx$ and is independent of $\sx$. This follows by using an approximation argument due equi-distribution of the orbit of $\sx$ and the fact that $T_{\alpha}$ is unique ergodicity, see \cite{KuipersNiederreiter1974}. Due to the structure of $T_{\alpha}$, this result is in line with those of \cite{StrungaruRichard2017}, where one takes $(\R,\TT,\Z\times\iota(\alpha)^\Z)$ as the cut and project scheme. Further, as we will shortly see in \eqref{weightsAreConvolution}, the weights $\Xi(T_\alpha,\cdot)$ are given by a convolution of certain functions.  This relation is emphasised in the following theorem.
\end{remark}

\begin{theorem} \label{thm:diffraction_rotation}
Let $\alpha\in \R^{+}$, $w \in [0, 1)$ and $f \in L^{1}([0, 1), \eta_{\alpha})$.  For a given $\sx \in [0, 1)$, let $\mu_{\sx}$ denote the $f$-weighted return time measure with respect to $T_{\alpha}$ and with reference point $\sx$.
\begin{enumerate}
\item\label{lem:4_1(a)} If $\alpha=p/q$ with $p, q \in \N$ and $\gcd(p,q) = 1$, then for every reference point $\sx\in\operatorname{supp}(\eta_{q,w})$, the diffraction of $\mu_\sx$ is given by
\begin{align*}
\widehat{\gamma_{\mu_\sx}}
 = \sum_{m\in\Z_q}\widehat{\Xi}(T_\alpha,\eta_{q,w})(m)\ \delta_{(\iota\alpha)^m}
 = \sum_{m \in\Z_q} \lvert \widehat{f_{\alpha, \sx}}\rvert^{2} (m) \ \delta_{(\iota\alpha)^m},
\end{align*}
where $f_{\alpha, \sx}(k) \coloneqq f(T_{\alpha}^{k}(\sx))$ for $k \in \Z_{q}$.
\item\label{lem:4_1(b)} 
If $\alpha\in\R^+\setminus\mathbb{Q}$, then the diffraction of $\mu_\sx$ is, for $\Lambda$-almost every reference points $\sx \in [0,1)$, given by
\begin{align*}
\widehat{\gamma_{\mu_\sx}} 
 = \sum_{m\in\Z}\widehat{\Xi}(T_\alpha,\Lambda)(m)\ \delta_{(\iota\alpha)^m}
 = \sum_{m \in\Z} \lvert \widehat{f_{\iota}}\rvert^{2}(m) \ \delta_{(\iota\alpha)^m},
\end{align*}
where $f_{\iota}(x) \coloneqq f(\iota^{-1} x)$ for $x \in \TT$. Additionally, if $f$ is Riemann integrable, then the statement holds for every reference point $\sx \in [0, 1)$.
\end{enumerate}
\end{theorem}

\begin{proof}
First we show Part \eqref{lem:4_1(a)}. By definition
\begin{align}\label{eq:theorem:autocorrelation_3}
\Xi(T_\alpha,\eta_{q,w})(k)
= \int f \circ T_{\alpha}^{-k}\cdot f\ \mathrm{d}\eta_{q,w}
= \int f_{\alpha, \sx}(l-k)\cdot f_{\alpha, \sx}(l)\ \mathrm{d}\omega_{\Z_q}(l) = f_{\alpha, \sx}*\widetilde{f_{\alpha, \sx}} (k),
\end{align}
and so $\widehat{\Xi}(T_\alpha,\eta_{q,w})(m) = (f_{\alpha, \sx}*\widetilde{f_{\alpha, \sx}})^\wedge(m) = \lvert\widehat{f_{\alpha, \sx}}\rvert^2(m)$.
We let $(\cdot, \cdot)_{\alpha} \colon \Z_{q} \times \Z_{q} \to \TT$ denote the character product induced by $T_{\alpha}$ and defined by $(k, z)_\alpha \coloneqq \exp(2\pi i\ \alpha\ k\ z)$, and let $(\cdot, \cdot) \colon \TT \times \Z_{q} \to \TT$ denote the character product defined by $(\iota x,z) \coloneqq \exp(2\pi i\  x\ z)$.  Letting $\varphi\in\mathcal{P}(\TT)$, we observe the following chain of equalities.
\begin{align*}
\sum_{z\in\Z} \Xi(T_\alpha,\eta_{q,w})(z)\ \widehat{\varphi}(z)
= \sum_{z\in\Z} \sum_{k\in\Z_q} \lvert\widehat{f_{\alpha,\sx}}\rvert^2(k)\ (k,z)_\alpha\ \widehat{\varphi}(z)
= \sum_{k\in\Z_q} \lvert\widehat{f_{\alpha,\sx}}\rvert^2(k) \sum_{z\in\Z} \widehat{\varphi}(z)\  ((\iota \alpha)^{k},z)
= \sum_{k\in\Z_q} \lvert\widehat{f_{\alpha,\sx}}\rvert^2(k)\ \varphi((\iota \alpha)^k)
\end{align*}
In the third equality we have used Lebesgue's dominated convergence theorem. For this, observe that the function $\lvert\widehat{f_{\alpha,\sx}}\rvert^2$ is bounded and that $\lim_{N\to\infty}\sum_{z=-N}^N \widehat{\varphi}(z)\  ((\iota\alpha)^k, z) = \varphi((\iota\alpha)^k)$, which is bounded for all $k\in\Z$, since $\varphi$ is a continuous function on $\TT$.  As a result
\begin{align*}
\lvert\widehat{f_{\alpha,\sx}}\rvert^2(k) \left\lvert\sum_{-N \leq z \leq N} \widehat{\varphi}(z)\  (\alpha k,z) \right\rvert
\leq \lvert\widehat{f_{\alpha,\sx}}\rvert^2(k) \!\!\!\! \sum_{-N \leq z \leq N} \!\! \lvert\widehat{\varphi}\rvert (z)
= \lvert\widehat{f_{\alpha,\sx}}\rvert^2(k) \!\!\!\! \sum_{-N \leq z \leq N} \!\! \widehat{\varphi}(z)
= \lvert\widehat{f_{\alpha,\sx}}\rvert^2(k) \!\!\!\! \sum_{-N \leq z \leq N} \!\! \widehat{\varphi}(z) (1, z)
\leq \lvert\widehat{f_{\alpha,\sx}}\rvert^2(k) \varphi(0),
\end{align*}
has a global bound, for all $k$ and $N$.  Hence $\langle \gamma_{\mu_\sx}, \widehat{\varphi} \rangle = \langle \widehat{\gamma_{\mu_\sx}}, \varphi \rangle$, which completes the proof of Part \eqref{lem:4_1(a)}.

Part \eqref{lem:4_1(b)} follows analogously to Part \eqref{lem:4_1(a)}, where one replaces $\eta_{\alpha,\sx}$ by $\eta_\alpha$, $(k,z)_\alpha$ by $((\iota \alpha)^k, z)$ and $f_{\alpha,\sx}$ by $f_{\iota}$. In particular,
\begin{align}
\label{eq:theorem:autocorrelation_4}
\Xi(T_{\alpha},\eta_\alpha)(k)
= \int f \circ T_{\alpha}^{-k}\cdot f\ \mathrm{d}\eta_\alpha
= \int f_\iota((\iota x) \cdot (\iota \alpha)^{-k})\cdot f_\iota(\iota x)\ \mathrm{d}\Lambda(x)
= f_\iota*\widetilde{f_\iota}((\iota\alpha)^k).
\end{align}
The final statement follows by using identical arguments to those given in \Cref{Rem:RiemannCase}.
\end{proof}

Here we emphasise that, for $w \in[0,1)$ and $\alpha \in \R^{+}$, when we write $f_{\iota}*\widetilde{f_{\iota}}((\iota \alpha)^{z})$, we mean the convolution with respect to $\Lambda$ evaluated at $(\iota \alpha)^{z} \in \TT$, for $z \in \Z$, and in the case that $\alpha=p/q$ with $p, q \in \N$ and $\operatorname{gcd}(p, q) = 1$, when we write $f_{\alpha,w}*\widetilde{f_{\alpha,w}}(k)$, we mean the convolution with respect to $\eta_{q,w}$ evaluated at $k \in \Z_{q}$, where $w \in [0, 1)$.  Namely, 
\begin{align}\label{weightsAreConvolution}
f_{\iota}*\widetilde{f_{\iota}}((\iota \alpha)^k)
= \int f\circ T_{\alpha}^{-k}\cdot f\ \mathrm{d}\Lambda
= \Xi(T_\alpha,\Lambda)(k)
\;\;\;\, \text{and} \;\;\;\,
f_{\alpha,w}*\widetilde{f_{\alpha,w}}(k)
= \int_{[0,1)} f\circ T_\alpha^{-k}\cdot f\ \mathrm{d}\eta_{\alpha,w}
= \Xi(T_\alpha,\eta_{q,w})(k).
\end{align}
For $\alpha \in \R^{+}$ and a sequence $(\alpha_i)_{i\in\N}$ in $\R^{+}$ which converges to $\alpha$, we require that the sequence $(\alpha_i)_{i\in\N}$ does not attain the value $\alpha$ infinitely often.  This condition is important as in the following lemma, if $\alpha = p/q$ with $p,q \in \N$ and $\operatorname{gcd}(p, q)=1$, we show that the pointwise limit of $\Xi(T_{\alpha_i}, \eta_{\alpha_{i}})$ does not coincide with $\Xi(T_{\alpha}, \eta_{q, w})$, for any $w \in [0, 1]$.  In forthcoming examples (\Cref{Example:1,Example:2}) it is shown that the chosen sequence of functions is the correct choice to guarantee convergence in \Cref{thm:convergence}.

\begin{lemma} \label{lem:eta_rotation_convergence}
Let $f\colon[0,1) \to \R^+$ be Riemann integrable, $\sx \in [0, 1)$, $\alpha \in \R^{+}$ and $(\alpha_i)_{i\in\N}$ be a sequence in $\R^{+}$ that converges to $\alpha$ with $\alpha_i\neq\alpha$ for $i \in \N$. When $\alpha_{i}$ is rational we let $p_{i}, q_{i} \in \N$ be such that $\operatorname{gcd}(p_{i}, q_{i}) = 1$ and $\alpha_{i} = p_{i}/q_{i}$. The sequence of functions given by $\Xi(T_{\alpha_i},\eta_{\alpha_i,y})$ if $\alpha_i$ is rational, and by $\Xi(T_{\alpha_i},\Lambda)$, if $\alpha_i$ is irrational, converge uniformly to $\Xi(T_\alpha,\Lambda)$.
\end{lemma}

\begin{proof}
In the case that $\alpha_{i} \not\in \mathbb{Q}$ for all $i \in \N$, the result is a consequence of the fact that  $f_\iota*\widetilde{f_\iota}\colon\TT\to\R$ is a continuous function on $\TT$ -- a compact space.  Assume that $\alpha_i \in \mathbb{Q}$ for all $i \in \N$.  Observe that
\begin{align}\label{eq:eq1_Lemma4.2}
\left\lvert f_\iota*\widetilde{f_\iota}((\iota\alpha)^z) - f_{\alpha_i,\sx}*\widetilde{f_{\alpha_i,\sx}}([z]_{q_i}) \right\rvert \leq& \left\lvert f_\iota*\widetilde{f_\iota}((\iota\alpha)^z) -f_{\iota}*\widetilde{f_{\iota}}((\iota\alpha_i)^z)\right\rvert + \left\lvert f_{\iota}*\widetilde{f_{\iota}}((\iota\alpha_i)^z) - f_{\alpha_i,\sx}*\widetilde{f_{\alpha_i,\sx}}([z]_{q_i}) \right\rvert.
\end{align}
The first first term on the right-hand-side of \eqref{eq:eq1_Lemma4.2} converges to zero, by an analogous argument to that given in the case when $\alpha_i \not\in \mathbb{Q}$ for all $i \in \N$.  To complete the proof, we show that the second term on the right-hand-side of \eqref{eq:eq1_Lemma4.2} also converges to zero.  Here, we will make use of the Riemann integrability of $f$.  Let $i \in \N$, $\sx_{i} = \min \Omega_{\alpha_{i}}(\sx)$ and $z \in \Z$.  Setting $l = [z]_{q_{i}}$ and
\begin{align*}
I_{m} \coloneqq
\begin{cases}
\left[\sx_{i} + \genfrac{}{}{}{}{m}{q_i}, \sx_{i} + \genfrac{}{}{}{}{m + 1}{q_i}\right] & \text{if} \; m \in \{ 0, 1, \dots, q_{i}-2\},\\[1em]
[0, \sx_{i}) \cup \left[\sx_{i} + \genfrac{}{}{}{}{q_{i}-1}{q_i}, 1\right] & \text{if} \; m = q_{i} - 1,
\end{cases}
\end{align*}
we have the following chain of inequalities. 
\begin{align*}
& \left\lvert f_{\iota}*\widetilde{f_{\iota}}((\iota\alpha_i)^z) - f_{\alpha_i,\sx_{i}}*\widetilde{f_{\alpha_i,\sx_{i}}}([z]_{q_i}) \right\rvert \nonumber \\
&= \left| \int f(x)\cdot f\left( \left\{ x-\genfrac{}{}{}{1}{l}{q_i} \right\} \right)\ \mathrm{d}\Lambda(x) - q_i^{-1} \sum_{m\in\Z_{q_i}} f\left( \left\{ \sx_{i}+\genfrac{}{}{}{1}{m}{q_i} \right\} \right)\cdot f\left( \left\{ \sx_{i}+\genfrac{}{}{}{1}{m-l}{q_i} \right\} \right) \right| \nonumber \\
&= \left| \sum_{m\in\Z_{q_i}} \int \chf{_{I_{_m}}}\!\!(x) \left( f(x)\cdot f\left( \left\{ x-\genfrac{}{}{}{1}{l}{q_i} \right\} \right) - f\left( \left\{ \sx_{i}+\genfrac{}{}{}{1}{m}{q_i} \right\} \right)\cdot f\left(\left\{ \sx_{i}+\genfrac{}{}{}{1}{m}{q_i}-\genfrac{}{}{}{1}{l}{q_i} \right\} \right) \right) \ \mathrm{d}\Lambda(x) \right| \nonumber \\
&= \left| \sum_{m\in\Z_{q_i}} \int \chf{_{I_{_m}}}\!\!(x) \left( f(x)\left(f\left( \left\{ x-\genfrac{}{}{}{1}{l}{q_i} \right\}\right) - f\left( \left\{ \sx_{i}+\genfrac{}{}{}{1}{m}{q_i}-\genfrac{}{}{}{1}{l}{q_i} \right\} \right) \right) + f\left( \left\{ \sx_{i}+\genfrac{}{}{}{1}{m}{q_i}-\genfrac{}{}{}{1}{l}{q_i} \right\} \right) \left( f(x)-f\left( \left\{ \sx_{i}+\genfrac{}{}{}{1}{m}{q_i} \right\} \right) \right) \right) \ \mathrm{d}\Lambda(x) \right| \nonumber \\
&= \left\lvert \sum_{m\in\Z_{q_i}} \int \chf{_{I_{[m -l]_{q_{i}}}}}\!\!\!\!(x) \cdot f\left( \left\{ x+\genfrac{}{}{}{1}{l}{q_i} \right\} \right)\left(f(x) - f\left(\left\{ \sx_{i}+\genfrac{}{}{}{1}{m-l}{q_i} \right\} \right) \right)\ \mathrm{d}\Lambda(x) + \!\! \int \chf{_{I_{_m}}}\!\!(x) \cdot f\left(\left\{\sx_{i}+\genfrac{}{}{}{1}{m - l}{q_i} \right\} \right) \left( f(x)-f\left( \left\{ \sx_{i}+\genfrac{}{}{}{1}{m}{q_i} \right\} \right) \right)\ \mathrm{d}\Lambda(x) \right\rvert \nonumber \\
&\leq \lVert f \rVert_{\infty}  \sum_{m\in\Z_{q_i}} q_i^{-1} \left( \sup \left\{ f(x) \colon x \in I_{[m -l]_{q_{i}}} \right\} - \inf \left\{ f(x) \colon x \in I_{[m -l]_{q_{i}}} \right\} + \sup \left\{ f(x) \colon x \in I_{m} \right\} - \inf \left\{ f(x) \colon x \in I_{m} \right\}  \right) \nonumber\\
&= 2\ \lVert f \rVert_{\infty} \sum_{m\in\Z_{q_i}} q_i^{-1} \left( \sup \left\{ f(x) \colon x \in I_{m} \right\} - \inf \left\{ f(x) \colon x \in I_{m} \right\} \right) 
\end{align*}
Since $\lim_{i \to \infty} q_i = \infty$, this latter term converges to zero by the Riemann property of Darboux.  Moreover, this latter term is independent of $z$, yielding uniform convergence.
\end{proof}

\begin{lemma} \label{lem:xi_varies}
Let $f \colon [0,1) \to \R^+_{0}$ be Riemann integrable, $\alpha \in [0,1)$ and $(\alpha_i)_{i \in \N}$ be a sequence in $\R^{+}$ such that $\lim_{i \to \infty} \alpha_{i} = \alpha$ with $\alpha_i \neq \alpha$, for all $i \in \N$.  Let $(\sx_i)_{i\in \N}$ denote a sequence of reference points in $[0, 1)$ and, for $i \in \N$, let $\mu_{\sx_{i}}$ denote the $f$-weighted return time measure with respect to $T_{\alpha_{i}}$ and with reference point $\sx_{i}$.  The sequence of autocorrelations $(\gamma_{\mu_{\sx_i}})_{i\in\N}$ attains a vague-limit $\gamma$ given by
\begin{align*}
\gamma = \sum_{z\in\Z} \Xi(T_\alpha,\Lambda)\ \delta_z.
\end{align*}
\end{lemma}

\begin{proof}
If $\alpha_i$ is irrational, since $f_\iota*\widetilde{f_\iota}$ is independent of the starting point $\sx_i$ the result is a direct consequence of \eqref{eq:theorem:autocorrelation_2}, \eqref{eq:theorem:autocorrelation_4} and \Cref{lem:eta_rotation_convergence}.  If $\alpha_i = p_{i}/q_{i}$ with $p_{i}, q_{i} \in \N$ and $\operatorname{gcd}(p_{i}, q_{i}) = 1$ for all $i \in \N$, then the result follows from \eqref{eq:theorem:autocorrelation_1}, \eqref{eq:theorem:autocorrelation_3} and an analogous argument as given in the proof of \Cref{lem:eta_rotation_convergence}.
\end{proof}

\begin{theorem}\label{thm:convergence}
Let $f\colon[0,1)\to\R^+_{0}$ be Riemann integrable and let $\alpha \in \R^{+}$.  Fix a sequence $(f_{i})_{i \in \N}$ of non-negative Riemann integrable functions which converge uniformly to $f$ on $[0, 1]$, and fix a sequence $(\alpha_i)_{i \in \N}$ in $\R^{+}$ with $\lim_{i \to \infty} \alpha_{i} = \alpha$ and $\alpha_i \neq \alpha$ for all $i \in \N$.  Let $(\sx_i)_{i\in \N}$ denote a sequence of reference points in $[0, 1)$ and, for $i \in \N$, let $\mu_{\sx_{i}}$ denote the $f_{i}$-weighted return time measure with respect to $T_{\alpha_{i}}$ and with reference point $\sx_{i}$.  The sequence of autocorrelations $(\gamma_{\mu_{\sx_i}})_{i\in\N}$ attains a vague-limit $\gamma$ given by
\begin{align*}
\gamma = \sum_{z\in\Z} \Xi(T_\alpha,\Lambda)(z)\ \delta_z.
\end{align*}
Hence, by \Cref{thm:diffraction_rotation}\eqref{lem:4_1(b)},
\begin{align*}
\widehat{\gamma} = \sum_{m\in\Z} \widehat{\Xi}(T_\alpha,\Lambda)(m)\ \delta_{(\iota \alpha)^{m}}.
\end{align*}
\end{theorem}

This result also holds for complex-valued functions by using the definition of Riemann integration of complex-valued functions as given in \cite{AmannEscher1999_Analysis}.  Further, \Cref{fig2} illustrates that the diffractions of two $f$-weighted return time measures associated to rigid rotations with irrational rotation numbers close together, are not too dissimilar in the vague topology.

\begin{figure}
  \centering
    \includegraphics[width=0.625\textwidth]{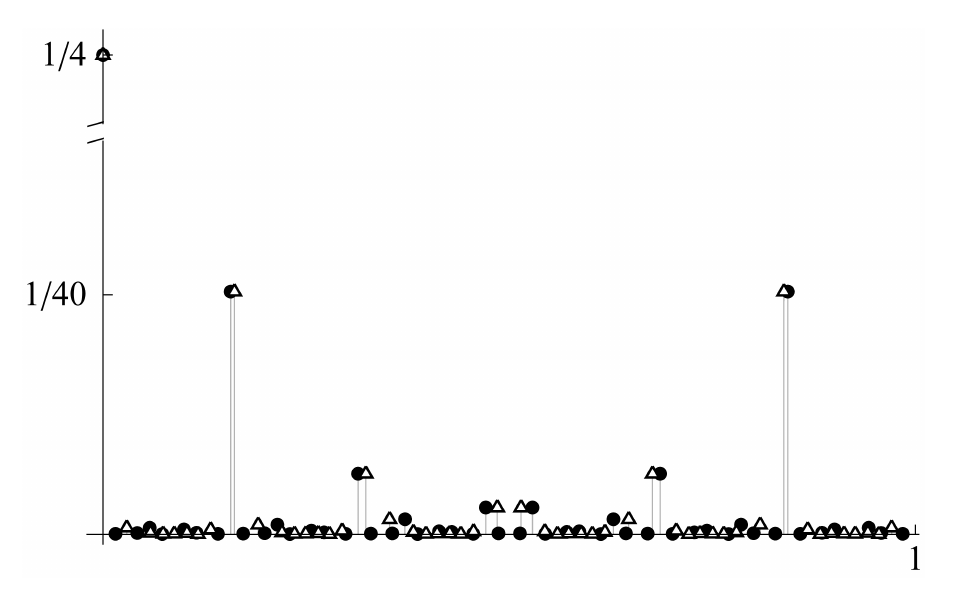}
      \caption{The $50$ largest atoms of the pure-point diffraction measure $\widehat{\gamma}$ for the  $f$-weighted return time measures with $f \colon [0,1]\to \R$ give by $f(x) = x$, and rotation numbers $\alpha=\pi /20$ (dots) and $\alpha =103 \pi/2000$ (triangles).}\label{fig2}
\end{figure}

\begin{proof}[Proof of \Cref{thm:convergence}]
If all $\alpha_i \not\in \mathbb{Q}$ for all $i \in \N$, then the convergence follow from \eqref{eq:theorem:autocorrelation_2}, \eqref{eq:theorem:autocorrelation_4} and the fact that $(f_i)_{\iota}*\widetilde{(f_i)_{\iota}}$ converges to $f_\iota*\widetilde{f_\iota}$ uniformly -- this fact is a direct consequence of how the involved maps are defined.

If $\alpha_i = p_{i}/q_{i}$ with $p_{i}, q_{i} \in \N$ and $\operatorname{gcd}(p_{i},q_{i}) =1$ for all $i \in \N$, then as $f_{i}$ converges to $f$ uniformly, given $\varepsilon>0$ there exists $N\in\N$ such that $\lVert f_i-f \rVert_\infty < \varepsilon$ for all $i\geq N$, and so letting $i \geq \N$,
\begin{align*}
\left\lvert f_{\iota}*\widetilde{f}_{\iota} - (f_{i})_{\alpha_i,\sx_{i}}*\widetilde{(f_{i})_{\alpha_i,\sx_{i}}} \right\rvert 
=& \left\lvert f_{\iota}*\widetilde{f}_{\iota} - f_{\alpha_i,\sx_i}*\widetilde{f_{\alpha_i,\sx_i}} + f_{\alpha_i,\sx_i}*\widetilde{f_{\alpha_i,\sx_i}} - (f_i)_{\alpha_i,\sx_i}*\widetilde{f_{\alpha_i,\sx_i}} + (f_i)_{\alpha_i,\sx_i}*\widetilde{f_{\alpha_i,\sx_i}} - (f_{i})_{\alpha_i,\sx_{i}}*\widetilde{(f_{i})_{\alpha_i,\sx_{i}}} \right\rvert \\
\leq& \left\lvert f_{\iota}*\widetilde{f_{\iota}} - f_{\alpha_i,\sx_i}*\widetilde{f_{\alpha_i,\sx_i}}  \right\rvert + \left\lvert (f_{\alpha_i,\sx_i}-(f_i)_{\alpha_i,\sx_i})*\widetilde{f_{\alpha_i,\sx_i}} \right\rvert + \left\lvert (f_i)_{\alpha_i,\sx_i}*(\widetilde{f_{\alpha_i,\sx_i}}-\widetilde{{f_{i}}_{\alpha_i,\sx_{i}}}) \right\rvert \\
\leq& \left\lvert f_{\iota}*\widetilde{f_{\iota}} - f_{\alpha_i,\sx_i}*\widetilde{f_{\alpha_i,\sx_i}} \right\rvert + \frac{2}{q_i} \sum_{m\in\Z_{q_i}} (\lVert f \rVert_\infty + \varepsilon) \ \lVert f_i - f \rVert_\infty \\
\leq& \left\lvert f_{\iota}*\widetilde{f_{\iota}} - f_{\alpha_i,\sx_i}*\widetilde{f_{\alpha_i,\sx_i}} \right\rvert +  2 \ (\lVert f \rVert_\infty + \varepsilon) \ \varepsilon.
\end{align*}
This together with \eqref{eq:theorem:autocorrelation_1}, \eqref{eq:theorem:autocorrelation_3} and \Cref{lem:xi_varies} yields the result. 
\end{proof}

The following examples (\Cref{Example:1,Example:2}) show, for a sufficiently nice $f$ and for $\alpha=p/q$ with $p, q \in \N$ and $\operatorname{gcd}(p,q) =1$, one may have $\lim_{i \to \infty} \Xi(T_{\alpha_i},\Lambda)(z)\to \Xi(T_\alpha,\eta_{q,0})(z)$ for $z\in \N$, where $(\alpha_i)_{i\in\N}  \in [0, 1)^{\N}$ is such that $\lim_{i \to \infty} \alpha_{i} = \alpha$.

\begin{example}\label{Example:1}
Let $\alpha=p/q$ with $p, q \in \N$ and $\operatorname{gcd}(p,q) =1$.  If $f = \chf{[0,\genfrac{}{}{}{2}{r}{q})}$ for a fixed $r\in\Z_q$, then for any sequence $(\alpha_i)_{i\in\N}  \in [0, 1)^{\N}$ we have $\lim_{i \to \infty}\alpha_i = \alpha$ implies that, for all $z\in\Z$,
\begin{align*}
\lim_{i \to \infty} \Xi(T_{\alpha_i},\Lambda)(z) = \Xi(T_\alpha,\eta_{q,0})(z).
\end{align*}
This is due to the fact that, for $m\in\Z_q$,
\begin{align*}
\int q \cdot f \cdot \chf{[{\frac{m-1}{q}},\ \frac{m}{q})} \ \mathrm{d}\Lambda = f\genfrac{(}{)}{}{2}{m-1}{q}.
\end{align*}
Hence, for all $m\in\Z_q$,
\begin{align*}
\Xi(T_{\alpha},\Lambda)(m) = \Xi(T_\alpha,\eta_{q,0})(m).
\end{align*}
Since $f_{\iota}*\widetilde{f_\iota}$ is a continuous function, if $\lim_{i \to \infty}\alpha_i = \alpha$, then, for all $z\in\Z$,
\begin{align*}
\lim_{i \to \infty} \Xi(T_{\alpha_i},\Lambda)(z)
= f_\iota*\widetilde{f_\iota}((\iota \alpha)^z)
= f_\iota*\widetilde{f_\iota}(\iota \{ z \alpha \})
= f_{\alpha, 0}*\widetilde{f_{\alpha, 0}}( [z]_{q})
= \Xi(T_\alpha,\eta_{q,0})(z).
\end{align*}
\end{example}

\begin{example}\label{Example:2}
Let $\alpha = p/q < 1/2$ with $p, q \in \N$ and $\operatorname{gcd}(p,q) =1$.  If $f = \chf{[0,(2p+1)/(2q))}$, then for any sequence  of irrationals $(\alpha_i)_{i\in\N}$ in  $\R^{+}$ with $\lim_{i \to \infty} \alpha_i  = \alpha$ and $z\in \Z$, we have
\begin{align*}
\lim_{i \to \infty} \Xi(T_{\alpha_i},\Lambda)(z)
= f_{\iota}*\widetilde{f_{\iota}}((\iota \alpha)^z)
= \begin{cases}
\frac{2p+1}{q} - \{ \alpha z \} & \{ \alpha z \} \in[0,\genfrac{}{}{}{2}{2p+1}{2q}),\\
0 & \{ \alpha z \} \in[\genfrac{}{}{}{2}{2p+1}{2q},1-\genfrac{}{}{}{2}{2p+1}{2q}),\\
\frac{2p+1}{q}-\left(1 - \{ \alpha z \} \right) & \{ \alpha z \} \in[1-\genfrac{}{}{}{2}{2p+1}{2q},1).
\end{cases}
\end{align*}
On the other hand, letting $l = [z]_{q}$,
\begin{align*}
\Xi(T_\alpha,\eta_{q,0})(z)
= q^{-1} \sum_{k\in\Z_q} \chf{_{[0, (2p+1)/(2q) )}}\left( \genfrac{}{}{}{1}{k}{q} \right) \cdot \chf{_{[0, (2p+1)/(2q))}} \left( \left\{ \genfrac{}{}{}{1}{k-l}{q} \right\} \right)
\end{align*}
At $z=0$, we have $f_{\iota}*\widetilde{f_{\iota}}(0)=\frac{2p+1}{q}\neq\frac{p}{q}=f_{\alpha,0}*\widetilde{f_{\alpha,0}}(0)$.
\end{example}

\begin{remark}
Let $f \in C([0, 1))$, $\sx \in [0, 1)$, $\alpha \in \R^{+}$ and $(\alpha_i)_{i \in \N}$ denote a fixed sequence in $\R^{+}$ with $\alpha_i \neq \alpha$ for all $i \in \N$. Let $\mu_{\alpha_{i},\sx}$ denote the $f$-weighted return time measure with respect to $T_{\alpha_{i}}$ and with reference point $\sx$, and let $\mu_{\sx}$ denote the $f$-weighted return time measure with respect to $T_{\alpha}$ and with reference point $\sx$. We immediately see that $\lim_{i \to \infty} \mu_{\alpha_i,\sx} = \mu_{\alpha,\sx}$ if $\lim_{i \to \infty}\alpha_i = \alpha$. To have $\lim_{i \to \infty} \gamma_{\mu_{\alpha_i,\sx}} = \gamma_{\mu_{\alpha,\sx}}$, by \eqref{eq:theorem:autocorrelation_1}--\eqref{eq:theorem:autocorrelation_4} and \Cref{lem:xi_varies}, it is necessary to show pointwise convergence of the sequence of maps given by $\Xi(T_{\alpha_i},\eta_{q,\sx})$ if $\alpha_i$ is rational, and $\Xi(T_{\alpha_i},\Lambda)$ if $\alpha_i$ is irrational.  However, the previous examples show that this is not always the case.
\end{remark}

\section*{Acknowledgements}

Part of this work was completed while the authors were visiting the Mittag-Leffler institute as part of the research program \textsl{Fractal Geometry and Dynamics}. We are extremely grateful to the organisers and staff for their very kind hospitality, financial support and stimulating atmosphere.  The authors also wish to thank M.~Baake and N.~Strungaru for many interesting and insightful discussions.  Finally, the authors are grateful to MINTernational for providing financial support for research visits between Universit\"at Bremen and California Polytechnic State University.

\bibliographystyle{plain}
\bibliography{sources}

\end{document}